\documentclass{amsart}[11pt,a4paper]
\usepackage{amscd,amsfonts,amssymb,amsmath,enumerate,mathtools,pinlabel,float}
\usepackage{blindtext}
\usepackage{tikz}

\newtheorem{theorem}{Theorem}[section]
\newtheorem{proposition}[theorem]{Proposition}
\newtheorem{lemma}[theorem]{Lemma}
\newtheorem{corollary}[theorem]{Corollary}

\theoremstyle{definition}
\newtheorem{definition}[theorem]{Definition}

\newtheorem{remark}[theorem]{Remark}

\theoremstyle{plain}
\numberwithin{equation}{section}
\newtheorem{theorem*}{Theorem}
\newtheorem{proposition*}[theorem*]{Proposition}
\newtheorem{corollary*}[theorem*]{Corollary}

\def \Z{\mathbb Z}

\newcommand{\Mod}{\mathrm{Mod}}
\newcommand{\UMod}{\mathrm{UMod}}
\newcommand{\Homeo}{\mathrm{Homeo}}
\newcommand{\LMod}{\mathrm{LMod}}

\renewcommand{\S}{\mathbb{S}}
\newcommand{\I}{\mathcal{I}}
\renewcommand{\O}{\mathcal{O}}

\newcommand{\C}{\mathcal{C}}
\newcommand{\D}{\mathcal{D}}
\newcommand{\T}{\mathcal{T}}
\newcommand{\Sp}{\mathrm{Sp}}
\newcommand{\SL}{\mathrm{SL}}

\newcommand{\Stab}{\mathrm{Stab}}
\numberwithin{equation}{subsection}

\begin{document}
\title[Liftable mapping class groups of regular cyclic covers]{Liftable mapping class groups of \\
regular cyclic covers}

\author[N. Agarwal]{Nikita Agarwal}
\address{Department of Mathematics\\
Indian Institute of Science Education and Research Bhopal\\
Bhopal Bypass Road, Bhauri \\
Bhopal 462 066, Madhya Pradesh\\
India}
\email{nagarwal@iiserb.ac.in}
\urladdr{https://sites.google.com/view/nagarwal/home}

\author[S. Dey]{Soumya Dey}
\address{The Institute of Mathematical Sciences \\
IV Cross Road, CIT Campus \\
Taramani \\
Chennai 600 113 \\
Tamil Nadu, India.}
\email{soumya.sxccal@gmail.com}
\urladdr{https://sites.google.com/site/soumyadeymathematics/}

\author[N. K. Dhanwani]{Neeraj K. Dhanwani}
\thanks{The third author was supported by a UGC fellowship.}
\address{Department of Mathematics\\
Indian Institute of Science Education and Research Bhopal\\
Bhopal Bypass Road, Bhauri \\
Bhopal 462 066, Madhya Pradesh\\
India}
\email{nkd9335@iiserb.ac.in}

\author[K. Rajeevsarathy]{Kashyap Rajeevsarathy}
\address{Department of Mathematics\\
Indian Institute of Science Education and Research Bhopal\\
Bhopal Bypass Road, Bhauri \\
Bhopal 462 066, Madhya Pradesh\\
India}
\email{kashyap@iiserb.ac.in}

\urladdr{https://home.iiserb.ac.in/$_{\widetilde{\phantom{n}}}$kashyap/}


\begin{abstract}
Let $\Mod(S_g)$ be the mapping class group of the closed orientable surface of genus $g \geq 1$. For $k \geq 2$, we consider the standard $k$-sheeted regular cover $p_k: S_{k(g-1)+1} \to S_g$, and analyze the liftable mapping class group $\LMod_{p_k}(S_g)$ associated with the cover $p_k$. In particular, we show that $\LMod_{p_k}(S_g)$ is the stabilizer subgroup of $\Mod(S_g)$ with respect to a collection of vectors in $H_1(S_g,\Z_k)$, and also derive a symplectic criterion for the liftability of a given mapping class under $p_k$. As an application of this criterion, we obtain a normal series of $\LMod_{p_k}(S_g)$, which generalizes a well known normal series of congruence subgroups in $\SL(2,\Z)$. Among other applications, we describe a procedure for obtaining a finite generating set for $\LMod_{p_k}(S_g)$ and examine the liftability of certain finite-order and pseudo-Anosov mapping classes. 
\end{abstract}
\maketitle

\section{Introduction}
\label{sec:intro}
Let $S_g$ be the closed connected orientable surface of genus $g \ge 1$, and let $\Mod(S_g)$ be the mapping class group  of $S_g$. For $k \ge 1$ and $g_k := k(g-1)+1$,  let $p_k : S_{g_k} \to S_g$ be the standard regular $k$-sheeted cover of $S_g$ induced by a $\Z_k$-action on $S_{g_k}$ generated by a $2\pi/k$ free rotation of $S_{g_k}$, as shown in Figure~\ref{fig:Sgk_free_hom} below (for $k =8$). 
\begin{figure}[h]
		\labellist
		\small
		\pinlabel $2\pi/8$ at 285 310
		\pinlabel $b_1'$ at 120 288
		\pinlabel $a_1'$ at 155 245
		\pinlabel $b_{1,2}$ at 232 255
		\pinlabel $b_{1,g}$ at 296 261
		\pinlabel $a_{1,2}$ at 207 242
		\pinlabel $a_{1,g}$ at 270 242
		\pinlabel $a_1$ at 57 35
		\pinlabel $c_1$ at 90 55
		\pinlabel $b_1$ at 39 56
		\pinlabel $a_2$ at 122 35
		\pinlabel $b_2$ at 145 54
		\pinlabel $a_g$ at 218 35
		\pinlabel $b_g$ at 237 52
		\endlabellist
		\centering
		\includegraphics[width=50ex]{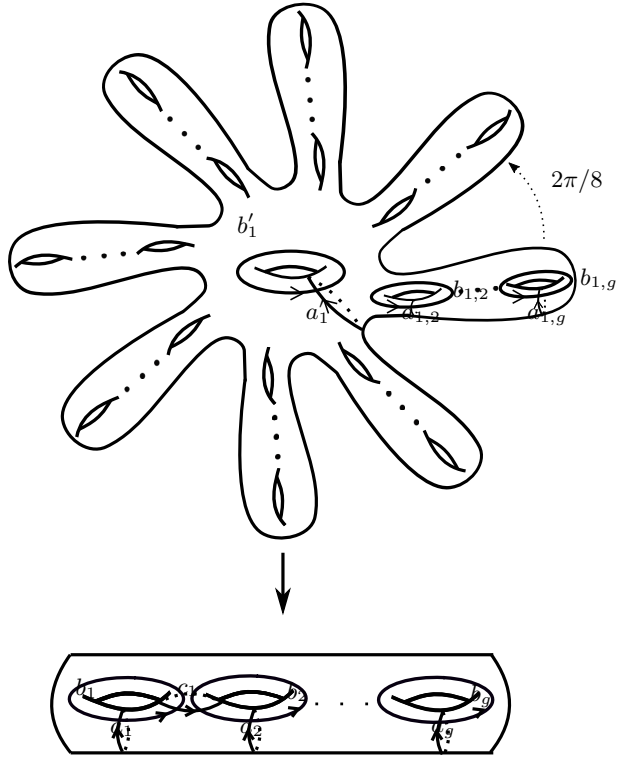}
		\caption{A $2\pi/8$ rotation of $S_{g_{8}}.$}
		\label{fig:Sgk_free_hom}
	\end{figure}
\noindent Let $\LMod_{p_k}(S_g)$ denote the liftable mapping class group of the cover $p_k$. The study of the liftable (and symmetric) mapping class groups was initiated by Birman-Hilden~\cite{BH1,BH2,BH3} in the 1970s with a view towards deriving presentations for mapping class groups.  More recently, these groups have been analyzed for balanced superelliptic and branched cyclic covers~\cite{GW2,GW1}, and braid groups~\cite{MP}. In our setup, $[\Mod(S_g):\LMod_{p_k}(S_g)] < \infty$, as this index is bounded above by the number of epimorphisms $H_1(S_g, \Z) \to \Z_k$, which is finite. Consequently, $\LMod_{p_k}(S_g)$ is finitely generated.

Let the standard generators of $H_1(S_{g_k},\Z)$ and $H_1(S_g,\Z)$ be represented by the the collections of simple closed curves
\begin{gather*} 
\{ a_1',b_1', a_{i,j}, b_{i,j} : 1 \le i \le k \text{ and } 2 \le j \le g \}
\text{ and } 
\{ a_i,b_i : 1 \le i \le g\}, 
\end{gather*} 
respectively, as indicated in Figure~\ref{fig:Sgk_free_hom} below. Let $f_{\ast}$ and $f_{\#}$ denote the homomorphisms induced by a map $f$ on the fundamental and first homology groups, respectively. By identifying our covering space $S_{g_k}$ with an appropriate subgroup $(p_k)_{\ast}(\pi_1(S_{g_k})) < \pi_1(S_g)$, we will assume that the induced map $(p_k)_{\#}: H_1(S_{g_k},\Z) \to H_1(S_g,\Z)$ is given by
\begin{equation}
\label{eqn:psharp}
\begin{aligned}
   a_{i,j} \mapsto a_j,& \text{ for } 1 \le i \le k \text{ and } 2 \le j \le g, \\
   b_{i,j} \mapsto b_j, &\text{ for } 1 \le i \le k \text{ and } 2 \le j \le g, \\
   a'_1 \mapsto a_1, & \text{ and } b'_{1} \mapsto b_1^k &.
\end{aligned}
\end{equation}
We will denote the homology classes $a_1,b_1, \ldots,a_g,b_g \in H_1(S_g,\Z) (\cong \Z^{2g})$ by the standard unit vectors $e_1, \ldots, e_{2g}$, respectively. Let $\Psi: \Mod(S_g) \to \Sp(2g, \Z)$ be the symplectic representation induced by the action of $\Mod(S_g)$ on $H_1(S_g, \Z)$, and let $\Psi_k$ denote the composition of the map $\Psi$ with the modulo $k$ reduction map $\Sp(2g, \Z) \to \Sp(2g, \Z_k)$. Since $\Mod(S_g)$ acts transitively on the nonseparating simple closed curves in $S_g$, it acts transitively on the primitive vectors in both $H_1(S_g,\Z)$ and $H_1(S_g,\Z_k)$. 

In Section~\ref{sec:sympl_crit}, we apply basic covering space theory to obtain a complete characterization of $\LMod_{p_k}(S_g)$ in terms of its image under both $\Psi$ and $\Psi_k$. 

\begin{theorem*}\label{thm:intro_main}
Given an $f \in \Mod(S_g)$, the following statements are equivalent.
\begin{enumerate}[(i)]
\item $f \in \LMod_{p_k}(S_g)$.
\item $f \in \Stab_{\Mod(S_g)}(\{\ell e_1:\ell \in \Z_k^{\times}\})$, where $e_1 \in H_1(S_g,\Z_k)$.
\item$\Psi(f) = (d_{ij})_{2g \times 2g}$, where $k|d_{2i}$, for $1\leq i\leq 2g$ and $i \neq 2$, and $d_{22} \in \Z_k^{\times}$.
\end{enumerate}
\end{theorem*}

\noindent The following corollary, which draws inspiration from \cite[Theorem 5.2.2]{TG}, is an immediate application of Theorem~\ref{thm:intro_main}.

\begin{corollary*}
For $g \geq 1$ and $k \geq 2$, we have 
$$\LMod_{p_k}(S_g) \cap \LMod_{p_{\ell}}(S_g) = \LMod_{p_d}(S_g),$$ where
$d = \text{lcm}(k,\ell)$.
\end{corollary*}

In Section~\ref{sec:appl}, we derive several other applications of Theorem~\ref{thm:intro_main}. It is classically known~\cite[\textsection 4.2]{TM} that there exists a normal series of congruence subgroups in $\SL(2,\Z)$ given by
\[ \tag{$\dagger$} 1 \lhd \Gamma(k) \lhd \Gamma_1(k) \lhd \Gamma_0(k), \text{ where } \]
\begin{gather*}
\Gamma(k) = \left\{ \begin{bmatrix} a & b \\ c & d \end{bmatrix} \in \SL(2,\Z) : a, d \equiv 1 \pmod{k}, \,b, c \equiv 0 \pmod{k}\right \},  \\
\Gamma_1(k) = \left\{ \begin{bmatrix} a & b \\ c & d \end{bmatrix} \in \SL(2,\Z) : a, d \equiv 1 \pmod{k}, \,c \equiv 0 \pmod{k}\right \},\\
\Gamma_0(k) = \left\{ \begin{bmatrix} a & b \\ c & d \end{bmatrix} \in \SL(2,\Z) :  \,c \equiv 0 \pmod{k}\right \},
\end{gather*}
and $\Gamma_0(k)/\Gamma_1(k) \cong \Z_k^{\times}$.
Since $\ker(\Psi_k) = \Mod(S_g)[k]$, the {\it level-$k$ subgroup} of $\Mod(S_g)$, it is apparent that $\Psi(\Mod(S_1)[k]) = \Gamma(k)$. We define $\Mod_{p_k}(S_g,e_1) := \Stab_{\Mod(S_g)}(e_1)$, where $e_1 \in H_1(S_g,\Z_k)$, and show that there exists a natural epimorphism $\Psi(\LMod_{p_k}(S_g)) \to \Z_k^{\times}$ whose kernel is $\Psi(\Stab_{\Mod(S_g)}(e_1))$. This yields an analog of the normal series ($\dagger$) in $\Sp(2g,\Z)$, for $g\geq 2$ given by
\[\tag{$\ast$} 1 \lhd \Psi(\Mod(S_g)[k]) \lhd \Psi(\Mod_{p_k}(S_g, e_1)) \lhd \Psi(\LMod_{p_k}(S_g)).\]
(It is worth mentioning here that in the theory of modular forms, the component groups of ($\ast$) are also known as \textit{paramodular groups}~\cite{HA,IO,RS}.) This series pulls back (as a consequence of Theorem~\ref{thm:intro_main}) via $\Psi$ to yield a generalization of ($\dagger$) to a normal series of subgroups in $\Mod(S_g)$, for $g \geq 2$. So, denoting a left-handed Dehn twist about a simple closed curve $c$ in $S_g$ by $T_c$, we have the following.

\begin{theorem*}\label{thm:ModSg-series_intro}
For $g \ge 1$ and $k \geq 2$, there exists a normal series of $\LMod_{p_k}(S_g)$ given by
$$ 1 \lhd \Mod(S_g)[k] \lhd \Mod_{p_k}(S_g, e_1) \lhd \LMod_{p_k}(S_g), \text{ where}$$ $$\LMod_{p_k}(S_g)/\Mod_{p_k}(S_g, e_1) \cong \Z_k^{\times} \text{ and }$$
$$[\Mod(S_g) : \Mod_{p_k}(S_g,e_1)] = | \{ \text{Primitive vectors in } \Z_k^{2g}\}|.$$ Moreover, the distinct cosets of $\LMod_{p_k}(S_g)/\Mod_{p_k}(S_g, e_1)$ are represented by the elements in $\S_k'' := \{T_{b_1}^{\bar{\ell}-1} T_{a_1}^{-1} T_{b_1}^{\ell -1} : \ell \in \Z_k^{\times} \text{ and } \ell \bar{\ell} \equiv 1 \pmod{k}\}$.
\end{theorem*}

Let $\iota \in \Mod(S_g)$ denote the hyperelliptic involution, and let $\phi$ denote the Euler totient function. By using Theorem~\ref{thm:intro_main} and the fact that $\Psi_k(\iota)$ commutes with every element of $\Sp(2g,\Z_k)$, we obtain the following extension of the normal series in Theorem~\ref{thm:ModSg-series_intro}.

\begin{corollary*}
\label{cor:norm_ser_ext_intro}
For $g \geq 1$ and $k \geq 3$, there exists a normal series of $\LMod_{p_k}(S_g)$ given by
$$ 1 \lhd \Mod(S_g)[k] \lhd \Mod_{p_k}(S_g, e_1) \lhd \langle \langle \Mod_{p_k}(S_g,e_1), \iota \rangle \rangle \lhd \LMod_{p_k}(S_g), \text{ where}$$ $\langle \langle \Mod_{p_k}(S_g,e_1), \iota \rangle \rangle$ denotes the normal closure of $\langle \Mod_{p_k}(S_g,e_1), \iota \rangle$ in $\LMod_{p_k}(S_g)$, 
$$\langle \langle \Mod_{p_k}(S_g,e_1), \iota \rangle \rangle/\Mod_{p_k}(S_g, e_1) \cong \Z_2 \text{ and }$$
$$[\LMod_{p_k}(S_g) : \langle \langle \Mod_{p_k}(S_g,e_1), \iota \rangle \rangle] = \phi(k)/2.$$
\end{corollary*}

\noindent As Corollary~\ref{cor:norm_ser_ext_intro} implies that $\LMod_{p_k}(S_g) \cong \langle \langle \Mod_{p_k}(S_g,e_1), \iota \rangle \rangle$ if and only if $\phi(k)=2$, we have the following. 

\begin{corollary*}
$\langle \langle \Mod_{p_k}(S_g,e_1) ,\iota \rangle \rangle = \LMod_{p_k}(S_g)$ if and only if $k \in \{3,4,6\}$.
\end{corollary*}

\noindent Let $B_n$ denote the braid group on $n$ strands, and let $S_1^1$ (resp. $S_{1,1}$) denote the surface $S_1$ with one boundary component (resp. one marked point). By taking the preimage of the component groups of the normal series in ($\dagger$) under the $cap$ epimorphism $\Mod(S_1^1) (\cong B_3) \xrightarrow{cap} \Mod(S_{1,1}) (\cong \SL(2,\Z))$, we obtain an analog of this series in $B_3$, whose component groups have quotients that are isomorphic to the corresponding quotients in the series ($\dagger$).

Let $\mathcal{L}_g= \{T_{a_1},T_{b_1}, \ldots,T_{a_g},T_{b_g},T_{c_1},\ldots,T_{c_{g-1}}\},$ be the Lickorish generating set~\cite{WBRL} for $\Mod(S_g)$ (where the curves are as indicated in Figure~\ref{fig:Sgk_free_hom}). Let $\S_{g,k}$ be a finite set of words in $\mathcal{L}_g$ generating $\Mod(S_g)[k]$. As finite generating sets for the congruence subgroups $\Gamma(k)$, $\Gamma_1(k)$, and $\Gamma_0(k)$ are well understood (see \cite[Lemma 12]{DKMO} and \cite[Proposition 1.17]{WS}), let $\S_k'$ be a finite set of words in the letters $\mathcal{L}_1$ generating $\Psi(\Mod_{p_k}(S_1,e_1)) = \Gamma_1(k)$ in $\Psi(\Mod(S_1))$ (as detailed in~\cite{WS}). Consider the corresponding set $\tilde{\S}_k$ of words (in $\mathcal{L}_1$) in $B_3$ of representatives of the distinct cosets of $cap^{-1}(\Gamma_1(k))/\ker(cap)$. For $g \geq 2$, let $i : S_1^1 \hookrightarrow S_g$ be an inclusion such that $\pi_1(i(S_1^1)) = \langle a_1,b_1 \rangle$, and let $\hat{i}:\Mod(S_1^1) \to \Mod(S_g)$ be the inclusion induced by $i$. By applying the symplectic criterion in Theorem~\ref{thm:intro_main} and by exploiting the structure of matrices in $\Sp(2g,\Z_k)$, we derive the following generating set for $\Mod_{p_k}(S_g,e_1)$.
\begin{theorem*}
\label{thm:lmod_gen_set_intro}
For $g,k \geq 2$, 
$$\Mod_{p_k}(S_g,e_1) = \langle \hat{i}(\tilde{\S}_k) \cup \S_{g,k} \cup \{T_{a_2},\ldots,T_{a_g},T_{b_2}, \ldots,T_{b_g},T_{c_1},\ldots,T_{c_{g-1}} \} \rangle.$$
\end{theorem*} 

\noindent Let $\I(S_g)$ denote the Torelli group of $S_g$. By applying Theorems~\ref{thm:ModSg-series_intro} and~\ref{thm:lmod_gen_set_intro}, we obtain generating sets for $\LMod_{p_k}(S_g)$, and hence $\displaystyle \UMod(S_g) :=  \bigcap_{k \geq 2} \LMod_{p_k}(S_g)$.

\begin{corollary*}
\label{cor:lmod_gen_set2_intro}
For $g,k \geq 2$, we have 
\begin{enumerate}[(i)]
\item $\LMod_{p_k}(S_g) = \langle \S_k'' \cup \hat{i}(\tilde{\S}_k) \cup \S_{g,k} \cup \{T_{a_2},\ldots,T_{a_g},T_{b_2}, \ldots,T_{b_g},T_{c_1},\ldots,T_{c_{g-1}} \} \rangle,$ and
\item $\UMod(S_g) = \langle \S(g) \cup \{T_{a_1},T_{a_2},\ldots,T_{a_g},T_{b_2}, \ldots,T_{b_g},T_{c_1},\ldots,T_{c_{g-1}}, \iota \} \rangle,$ where $\I(S_g) = \langle \S(g) \rangle$.
\end{enumerate}
\end{corollary*}

\noindent Thus, by a result due to Johnson~\cite{DJ}, we can now infer (from Corollary~\ref{cor:lmod_gen_set2_intro} (ii)) that $\UMod(S_g)$ is indeed finitely generated for $g \geq 3$. 

In the penultimate subsection of Section~\ref{sec:appl}, we derive conditions for the existence of conjugates of certain finite-order mapping class and roots of Dehn twists in $\Mod_{p_k}(S_g,e_1)$, by examining their reduction systems. In particular, we show that: 
\begin{proposition*}
For $g \geq 2$, if $f \in \Mod(S_g)$ is either a finite-order mapping class whose corresponding orbifold has genus strictly greater than zero, or a nontrivial root of a Dehn twist about a nonseparating curve in $S_g$, then $f$ has a conjugate in $\Mod_{p_k}(S_g,e_1)$.
\end{proposition*}

\noindent Further, we show that there exist no irreducible finite-order mapping class in $\Mod_{p_k}(S_g,e_1)$. We also derive conditions for the existence in $\Mod_{p_k}(S_g,e_1)$, of a conjugate of root of a Dehn twist about a separating curve (see Proposition~\ref{prop:root_lift}). In the final subsection, we derive equivalent conditions for the liftability (under $p_k$) of a certain class of Penner-type~\cite{RP} pseudo-Anosovs in $\Mod(S_g)$. More specifically, we show the following.

\begin{proposition*}
For $g \geq 2$, let $f\in \Mod(S_g)$ be a Penner-type pseudo-Anosov mapping class of the form
$$f := \prod_{i=1}^{g-1} T_{c_i}^{-p_i} \prod_{j=1}^g \left(T_{a_j}^{-k_j}T_{b_j}^{\ell_j}\right).$$
Then $f \in \LMod_{p_k}(S_g)$ if and only if $k \mid \ell_1$. 
\end{proposition*}

\section{Symplectic criteria for liftability}
\label{sec:sympl_crit}
In this section, we will derive an explicit description of the structure of an arbitrary matrix in $\Psi(\LMod_{p_k}(S_g))$. We will follow the notation introduced in Section~\ref{sec:intro}, and also assume from here on that $e_1 \in H_1(S_g,\Z_k)$ (unless mentioned otherwise). The key ingredient in the derivation of our symplectic criterion is the following result from the theory of group actions on surfaces~\cite{H1,M1}. 
\begin{theorem}
\label{thm:group_actions}
For $\ell \geq 1$, a finite group $G$ acts on a surface $S_{\ell}$ if and only if there exists a short exact sequence
$$1\to \pi_1(S_{\ell})\xrightarrow{\alpha} \pi_1^{orb}(S_{\ell}/G) \xrightarrow{\varphi} G \to 1,$$
where $\pi_1^{orb}(S_{\ell}/G)$ is the orbifold fundamental group of $S_{\ell}/G$ and $\varphi$ is an order-preserving epimorphism. In particular, for the cover $p_k : S_{g_k} \to S_g$, we have the short exact sequence:
$$1\to \pi_1(S_{g_k})\xrightarrow{(p_k)_{\ast}} \pi_1(S_g) \xrightarrow{\varphi} \Z_k \to 1.$$
\end{theorem}

\subsection{Criteria for $\LMod_{p_k}(S_g)$} Let $f \in \Mod(S_g)$. Using the fact that our covering action is cyclic, we obtain a homological criterion for the liftability of $f$ under $p_k$ which, in turn, yields an equivalent symplectic criterion for $f$ to be contained in $\LMod_{p_k}(S_g)$.

\begin{theorem}\label{thm:symp_lift_crit}
Given an $f \in \Mod(S_g)$, the following statements are equivalent.
\begin{enumerate}[(i)]
\item $f \in \LMod_{p_k}(S_g)$.
\item $\Psi(f) = (d_{ij})_{2g \times 2g}$, where $k|d_{2i}$, for $1\leq i\leq 2g$ and $i \neq 2$, and $\gcd(d_{22},k)=1$.
\item $\Psi_k(f) = (e_{ij})_{2g \times 2g}$, where $e_{2i}=0$, for $1\leq i\leq 2g$ and $i \neq 2$, and $e_{22} \in \Z_k^{\times}$.
\item $f \in \Stab_{\Mod(S_g)}(\{\ell e_1:\ell \in \Z_k^{\times}\})$.
\end{enumerate}
\end{theorem}
\begin{proof}
First, we show that (i)$\iff$(ii). Consider an epimorphism $\varphi:\pi_1(S_{g}) \to \mathbb{Z}_k$ such that $\ker{(\varphi)} = (p_k)_{\ast}(\pi_1(S_{g_k}))$. By Theorem~\ref{thm:group_actions} and basic covering space theory, it follows that an $f \in \LMod(S_g)$ if and only if $f$ is represented by an $F \in \text{Homeo}^+(S_g)$ such that $F_{\ast}$ preserves the conjugacy class of $(p_k)_{\ast}(\pi_1(S_{g_k})) = \ker(\varphi)$. Since $\Z_k$ is abelian, this is equivalent to the condition that $F_{\#}(\ker{{\bar{\varphi}}})\subseteq \ker(\bar{\varphi})$, where $\bar{\varphi}$ is the induced map on the abelianization $H_1(S_g,\Z)$ of $\pi_1(S_g)$. Finally, from the definition of $(p_k)_{\#}$ (in Equation~(\ref{eqn:psharp}) in the preceding discussion), it is now clear that
$$\ker(\bar{\varphi}) = \langle e_1,k e_2,\dots,e_{2g-1},e_{2g}\rangle,$$
which proves our assertion.

The assertion (ii)$\iff$(iii) is apparent. We now show that $(i)\iff(iv)$ to complete the argument. We consider the following series of correspondences between the finite sets:
\begin{gather*}
\mathcal{C}_k = \{\text{Regular $k$-sheeted covers of $p_k: S_{g_k}\to S_g$ whose deck transformation group is $\Z_k$}\} \\
\updownarrow \\
\{\text{Endomorphisms $(p_k)_{\ast}:\pi_1(S_{g_k}) \hookrightarrow \pi_1(S_g)$ such that $\pi_1(S_g)/(p_k)_{\ast}(\pi_1(S_{g_k})) \cong \Z_k$}\} \\
\updownarrow \\
\{\{\text{Epimorphisms } \varphi :\pi_1(S_g) \to \Z_k : (p_k)_{\ast}(\pi_1(S_{g_k})) = \ker({\varphi})\}: p_k \in \mathcal{C}_k\} \\
\updownarrow \\
\{\{\text{Epimorphisms } \bar{\varphi} :H_1(S_g) \to \Z_k : (p_k)_{\#}(H_1(S_{g_k})) = \ker(\bar{\varphi})\}: p_k \in \mathcal{C}_k\} \\
\updownarrow \\
\{\{\ell v : \ell \in \Z_k^{\times} \}:  v \in H_1(S_g,\Z_k) \text{ is primitive}\}
\end{gather*}
Since $\Mod(S_g)$ acts transitively on primitive vectors in $H_1(S_g, \Z_k)$, it follows that it acts transitively on the covers in $\mathcal{C}_k$. Hence, for our cover $p_k$ (which induces $(p_k)_{\#}$ as described in Equation~\ref{eqn:psharp}), we have
$$\LMod_{p_k}(S_g) = \Stab_{\Mod(S_g)}(p_k) = \Stab_{\Mod(S_g)}(\{\ell e_1 : \ell \in \Z_k^{\times} \}),$$ and our assertion follows.
\end{proof}

\noindent We recall the notation from Section~\ref{sec:intro} that for $k \geq 2$, $\Mod(S_g)[k] = \ker(\Psi_k)$, the level-$k$ subgroup of $\Mod(S_g)$. We have the following immediate consequence of Theorem~\ref{thm:symp_lift_crit}.
\begin{corollary}\label{cor:torelli_lift}
For $g \geq 1$ and $k \geq 2$, we have $\Mod(S_g)[k] \lhd \LMod_{p_k}(S_g)$.
\end{corollary}

\noindent The next corollary follows immediately from Theorem~\ref{thm:symp_lift_crit} (ii).

\begin{corollary}
For $g \geq 1$ and $k,\ell \geq 2$, we have
$$\LMod_{p_k}(S_g) \cap \LMod_{p_{\ell}}(S_g) = \LMod_{p_d}(S_g),$$ where
$d = \text{lcm}(k,\ell)$.
\end{corollary}

\noindent Let $\displaystyle \UMod(S_g) := \bigcap_{k \geq 2} \LMod_{p_k}(S_g).$ From Corollary~\ref{cor:torelli_lift} and the fact that $\I(S_g) \lhd \Mod(S_g)[k]$, for each $k \geq 2$, it follows that the Torelli group $\I(S_g) \subset \UMod(S_g)$, for $g \ge 1$, and hence $\UMod(S_g) \neq \emptyset$. This leads us to another direct consequence of Theorem~\ref{thm:symp_lift_crit}.
\begin{corollary}
	Let $f \in \Mod(S_g)$, and let $\Psi(f) = (d_{ij})_{2g \times 2g}$. Then, $f \in \UMod(S_g)$ if and only if $d_{2i} = 0$, for $1\leq i\leq 2g$ and $i \neq 2$, and $d_{22}= \pm 1$.
\end{corollary}

\noindent Let $\phi$ denote the Euler totient function. The arguments in the proof of Theorem~\ref{thm:symp_lift_crit} together with Corollary~\ref{cor:torelli_lift} yield the following.

\begin{corollary}\label{cor:index_lmod}
Let $I_g = \Psi(\LMod_{p_k}(S_g))$ and $I_{g,k} = \Psi_k(\LMod_{p_k}(S_g))$. Then
\begin{enumerate}[(i)]
\item $\displaystyle [\Mod(S_g) : \LMod_{p_k}(S_g) ] = [\Sp(2g, \Z) : I_g ] = [\Sp(2g, \Z_k) : I_{g,k} ]$ and 
\item $\displaystyle [\Mod(S_g) : \LMod_{p_k}(S_g) ] = \frac{|\{\text{Primitive elements in }\Z_k^{2g} \}|}{\phi(k)}$. In particular, when $k$ is prime,
$$\displaystyle [\Mod(S_g) : \LMod_{p_k}(S_g) ] = \frac{k^{2g}-1}{k-1}.$$
\end{enumerate}
\end{corollary}

\subsection{Criterion for $\Mod_{p_k}(S_g,e_1)$} Recalling from Section~\ref{sec:intro} that $\Mod_{p_k}(S_g,e_1) = \Stab_{\Mod(S_g)} (e_1)$, we have the following.

\begin{theorem}\label{thm:symp_Gamma1_crit}
Given an $f \in \Mod(S_g)$, the following statements are equivalent.
\begin{enumerate}[(i)]
\item $f \in \Mod_{p_k}(S_g,e_1)$.
\item If $\Psi(f) = (d_{ij})_{2g \times 2g}$, then $k|d_{i1}$ and $k | d_{2j}$, for $1\leq i,j\leq 2g$ when $i \ne 1$ and $j \ne 2$, and $d_{11}, d_{22} \equiv 1 \pmod{k}$.
\item If $\Psi_k(f) = (e_{ij})_{2g \times 2g}$, then $e_{11} = e_{22} = 1$ and $e_{i1} = e_{2j} = 0$ for $i \ne 1$ and $j \ne 2$.
\end{enumerate}
\end{theorem}

\begin{proof}
 Let $D = (d_{ij})_{2g \times 2g} \in \Sp(2g, \Z)$. Let $D'$ be the matrix $D$ modulo $k$. Since $D'(e_1)=e_1$, it follows that $[ \, d_{11} \, d_{21} \ldots  \,d_{(2g)1} \, ]^T = [ \, 1 \, 0 \ldots \,0 \, ]^T \pmod{k}.$ Moreover, $D$ preserves the bilinear form
$$J = \begin{bmatrix}
J_1 & O_2 & \ldots & O_2 \\
O_2 & J_1 & \ldots & O_2 \\
\vdots& \vdots  &  & \vdots \\
O_2 & O_2 & \ldots & J_1
\end{bmatrix},$$ 
where $J_1 = \begin{bmatrix}
0 & 1 \\
-1 & 0 \,
\end{bmatrix}$ and $O_2$ is the $2 \times 2$ zero matrix. Thus, we have
 $$ D' = \begin{bmatrix}
1&d_{12}& d_{13}&\ldots&d_{1(2g)}\\
0&1&0&\ldots& 0\\
\vdots&\vdots&\vdots&&\vdots\\
0&d_{(2g)2}&d_{(2g)3} & \ldots&d_{(2g)(2g)}\\
\end{bmatrix}. $$
 This shows (i)$\iff$(iii). The equivalence (ii)$\iff$(iii) is clear.
\end{proof}

\noindent We will use the following fact from basic group theory in the proof of the next corollary. 
\begin{lemma}
\label{lem:group_theory_fact}
If $\psi : G\to H$ is a group epimorphism and $N < G$ such that $N = \psi^{-1}(\psi(N))$, then $N \lhd G \iff \psi(N) \lhd H.$
\end{lemma}

\begin{corollary}
\label{cor:index_gamma1_in_lmod}
For $g \geq 1$ and $k \geq 2$, we have
\begin{enumerate}[(i)]
\item  $\Mod(S_g)[k] \lhd \Mod_{p_k}(S_g,e_1) \lhd \LMod_{p_k}(S_g)$, and
\item $\LMod_{p_k}(S_g)/\Mod_{p_k}(S_g,e_1) \cong \mathbb{Z}_k^{\times}$. Consequently, $$[\LMod_{p_k}(S_g):\Mod_{p_k}(S_g,e_1)] = \phi(k).$$
\end{enumerate}
\end{corollary}

\begin{proof}
From Theorem~\ref{thm:symp_Gamma1_crit}, it follows directly that $\Mod(S_g)[k] \lhd \Mod_{p_k}(S_g,e_1)$, for $k \geq 2$. Further, by definition of $\LMod_{p_k}(S_g)$, it is clear that
$\Mod_{p_k}(S_g,e_1) < \LMod_{p_k}(S_g)$. So it remains to show that $\Mod_{p_k}(S_g,e_1)$ is normal in $\LMod_{p_k}(S_g)$. By Lemma~\ref{lem:group_theory_fact}, this is equivalent to showing that $\Psi_k(\Mod_{p_k}(S_g,e_1)) \lhd \Psi_k(\LMod_{p_k}(S_g))$.

Given an $f \in \LMod_{p_k}(S_g)$, we showed in Theorem~\ref{thm:symp_lift_crit}, that $\Psi_k(f) = E = (e_{ij})_{2g \times 2g}$, where $e_{2i}=0$, for $1\leq i\leq 2g$ and $i \neq 2$, and $e_{22} \in \Z_k^{\times}$. As $E$ preserves the symplectic form $J$ (in the proof of Theorem~\ref{thm:symp_Gamma1_crit}), it further simplifies to a matrix of the form
$$E = \begin{bmatrix}
e_{11} & e_{12} & e_{13} & \ldots & e_{1(2g)} \\
0 & e_{22} & 0 & \ldots &  0 \\
\vdots & \vdots & \vdots & & \vdots \\
0 & e_{(2g)2} & e_{(2g)3} & \ldots & e_{(2g)(2g)}
\end{bmatrix},$$ where $e_{22} \in \Z_k^{\times}$.
Considering the epimorphism $\alpha : \Psi_k(\LMod_{p_k}(S_g)) \to \Z_k^{\times} : E \mapsto e_{22}$, we see that $\ker(\alpha)  = \Psi_k(\Mod_{p_k}(S_g,e_1))$. Therefore, the normality of $\Mod_{p_k}(S_g,e_1)$ and (ii) will now follow from Theorem~\ref{thm:symp_Gamma1_crit}.
\end{proof}

\section{Applications}
\label{sec:appl}
In this section, we will derive several applications of Theorem~\ref{thm:symp_lift_crit}.
\subsection{Normal series' of subgroups of $\Mod(S_g)$ and $B_3$} Putting together the results in Corollaries~\ref{cor:index_lmod} and~\ref{cor:index_gamma1_in_lmod}, we obtain the following generalization of classical normal series ($\dagger$) of congruence subgroups in $\SL(2,\Z)$.
\begin{theorem}\label{thm:norm_series_gen}
For $g \ge 1$ and $k \geq 2$, there exists a normal series of $\LMod_{p_k}(S_g)$ given by
$$ 1 \lhd \Mod(S_g)[k] \lhd \Mod_{p_k}(S_g, e_1) \lhd \LMod_{p_k}(S_g), \text{ where}$$ $$\LMod_{p_k}(S_g)/\Mod_{p_k}(S_g, e_1) \cong \Z_k^{\times} \text{ and }$$
$$[\Mod(S_g) : \Mod_{p_k}(S_g,e_1)] = | \{ \text{Primitive vectors in } \Z_k^{2g}\}|.$$
\end{theorem}

Let $S_{1,1}$ denote $S_1$ with one marked point and let $S_1^1$ denote $S_1$ with one boundary component. Let $B_n$ denote the braid group on $n$ strands. Since $\Mod(S_{1,1}) \cong \Mod(S_1)$ and $\Mod(S_1^1) \cong B_3$, the natural capping map $S_1^1 \rightarrow S_{1,1}$ induces an epimorphism $cap: \Mod(S_1^1) \to \Mod(S_1)$, which yields an exact sequence: 
$$1 \to \langle T_s \rangle \to B_3 \xrightarrow{cap} \Mod(S_1) \to 1,$$ where $s$ represents the isotopy class of $\partial S_1^1$. Pulling back the classical normal series ($\dagger$) (i.e the normal series in Theorem~\ref{thm:norm_series_gen} for $g=1$) under the $cap$ homomorphism, we obtain an analogous normal series of subgroups in $B_3$.
\begin{corollary}\label{cor:norm_series_B3}
For $k \geq 2$, there exists a normal series of $cap^{-1}(\Gamma_0(k))$ in $B_3$ given by
$$1 \lhd cap^{-1}(\Gamma(k)) \lhd  cap^{-1}(\Gamma_1(k)) \lhd cap^{-1}(\Gamma_0(k)), \text{ where}$$
$$cap^{-1}(\Gamma_0(k))/cap^{-1}(\Gamma_1(k)) \cong \Z_k^{\times} \text{ and }$$
$$[B_3 : cap^{-1}(\Gamma_1(k))] = | \{ \text{Primitive vectors in } \Z_k^2\}|.$$
\end{corollary}
In the theory of modular forms, finite generating sets for the congruence subgroups $\Gamma(k)$, $\Gamma_1(k)$, and $\Gamma_0(k)$ are well understood (see ~\cite[Lemma 12]{DKMO} for an explicit generating set for $\Gamma_0(k)$ and~\cite[Proposition 1.17]{WS} for a computational approach for deriving generators for any congruence subgroup). Since $B_3/\langle T_s \rangle \cong \Mod(S_1)$, by including $T_s$ into these generating sets, we can obtain finite generating sets for all the component groups appearing in the normal series in Corollary ~\ref{cor:norm_series_B3}. 

It follows from Theorem~\ref{thm:symp_Gamma1_crit} and Corollary~\ref{cor:index_gamma1_in_lmod} that the epimorphism $\Psi_k$ induces an 
isomorphism $$\bar{\Psi}_k : \LMod_{p_k}(S_g)/ \Mod_{p_k}(S_g,e_1) \to \Psi_k(\LMod_{p_k}(S_g))/\Psi_k(\Mod_{p_k}(S_g,e_1)),$$ where the quotient groups on either side are isomorphic to $\Z_k^{\times}$. It is apparent that for each $\ell \in \Z_k^{\times}$ (with multiplicative inverse $\bar{\ell}$), the matrix $Q_{\ell} = ((q_{\ell})_{ij})_{2g \times 2g} \in \Sp(2g,\Z_k)$, where
$$(q_{\ell})_{ij} = \begin{cases}
\ell, & \text{if } i = j=1,\\
\bar{\ell}, & \text{if } i =  j=2, \\
1, & \text{if } i = j > 2, \text{ and}\\
0, &\text{otherwise,}
\end{cases}$$
represents a coset of $\Psi_k(\LMod_{p_k}(S_g))/\Psi_k(\Mod_{p_k}(S_g,e_1))$. Moreover, each coset represented by a $Q_{\ell}$ pulls back under $\bar{\Psi}_k$ to a coset represented by 
$\varphi_{\ell} = T_{b_1}^{\bar{\ell}-1} T_{a_1}^{-1}T_{b_1}^{\ell-1}$ in $\LMod_{p_k}(S_g)/ \Mod_{p_k}(S_g,e_1)$. Thus, we have the following corollary. 

\begin{corollary}
\label{cor:quot_lmod}
For $g \geq 1$ and $k \geq 2$, the distinct cosets of $\LMod_{p_k}(S_g)/\Mod_{p_k}(S_g, e_1)$ are represented by the elements in 
$$\S_k'' := \{\varphi_\ell : \ell \in \Z_k^{\times} \}.$$
\end{corollary}

For $g \geq 1$, let $\iota \in \Mod(S_g)$ denote the hyperelliptic involution, and let $I_{2g} \in \Sp(2g,\Z_k)$ be the identity matrix. Since $\Psi_k(\iota) = -I_{2g}$, it follows from Theorem~\ref{thm:symp_Gamma1_crit} that $\iota \in \Mod_{p_k}(S_g, e_1)$ if and only if $k=2$. Moreover, for $k=2$, it is apparent from Theorem~\ref{thm:norm_series_gen} that $\LMod_{p_k}(S_g) = \Mod_{p_k}(S_g,e_1)$. Thus, as $\Psi_k(\iota)$ commutes with every element of $\Sp(2g,\Z_k)$, we obtain the following extension of the normal series in Theorem~\ref{thm:norm_series_gen}. 

\begin{corollary}
\label{cor:norm_ser_ext}
For $g \geq 1$ and $k \geq 3$, there exists a normal series of $\LMod_{p_k}(S_g)$ given by
$$ 1 \lhd \Mod(S_g)[k] \lhd \Mod_{p_k}(S_g, e_1) \lhd \langle \langle \Mod_{p_k}(S_g,e_1), \iota \rangle \rangle \lhd \LMod_{p_k}(S_g), \text{ where}$$ $\langle \langle \Mod_{p_k}(S_g,e_1), \iota \rangle \rangle$ denotes the normal closure of $\langle \Mod_{p_k}(S_g,e_1), \iota \rangle$ in $\LMod_{p_k}(S_g)$, 
$$\langle \langle \Mod_{p_k}(S_g,e_1), \iota \rangle \rangle/\Mod_{p_k}(S_g, e_1) \cong \Z_2 \text{ and }$$
$$[\LMod_{p_k}(S_g) : \langle \langle \Mod_{p_k}(S_g,e_1), \iota \rangle \rangle] = \phi(k)/2.$$
\end{corollary}

\noindent From Corollary~\ref{cor:norm_ser_ext}, it follows that $\LMod_{p_k}(S_g) \cong \langle \langle \Mod_{p_k}(S_g,e_1), \iota \rangle \rangle$ if and only if $\phi(k)=2$. Hence, we have the following. 

\begin{corollary}
$\langle \langle \Mod_{p_k}(S_g,e_1) ,\iota \rangle \rangle = \LMod_{p_k}(S_g)$ if and only if $k \in \{3,4,6\}$.
\end{corollary}

\subsection{Generating $\Mod_{p_k}(S_g,e_1)$ and $\LMod_{p_k}(S_g)$}
Let $T_c$ denote the left-handed Dehn twist about a simple closed curve $c$ in $S_g$. The Lickorish generating set for $\Mod(S_g)$ is given by the set of Dehn twists
$$\mathcal{L}_g= \{T_{a_1},T_{b_1}, \ldots,T_{a_g},T_{b_g},T_{c_1},\ldots,T_{c_{g-1}}\},$$ where the curves $a_i,b_i$, for $1 \leq i \leq g$, and $c_j$, for $1 \leq j \leq g-1$ are as indicated in Figure~\ref{fig:Sgk_free_hom} in Section~\ref{sec:intro}. Since $[\Mod(S_g) : \Mod(S_g)[k]] < \infty$ for $k \geq 2$, $\Mod(S_g)[k]$ is clearly finitely generated. However, explicit generating sets for $\Mod(S_g)[k]$ are not known, except when $k=2$ (see~\cite{NJF} and references therein). So, for our purposes, we shall simply assume the existence of a finite set $\S_{g,k}$ of words in $\mathcal{L}_g$ generating $\Mod(S_g)[k]$. Let $i : S_1^1 \hookrightarrow S_g$ be an inclusion such that $\pi_1(i(S_1^1)) = \langle a_1,b_1 \rangle$. Then $i$ induces an inclusion $\hat{i}: \Mod(S_1^1) (\cong B_3) \hookrightarrow \Mod(S_g)$, for $g \geq 2$. Let $cap:B_3 = \langle T_{a_1},T_{b_1} \rangle \to \Mod(S_1) = \langle T_{a_1'},T_{b_1'} \rangle$ be the epimorphism (described in Corollary~\ref{cor:norm_series_B3}) which maps the twists $T_{a_1}\text{ and }T_{b_1}$ to the twists $T_{a_1'} \text{ and }T_{b_1'}$, respectively. Let $\S_k' = \{w_1(T_{a_1'},T_{b_1'}), \ldots, w_{n_k}(T_{a_1'},T_{b_1'})\}$ be a finite set of words in the letters $T_{a_1'},T_{b_1'}$ generating $\Mod_{p_k}(S_1,e_1) = \Gamma_1(k)$ in $\Mod(S_1)$ (as detailed in~\cite{WS}). We consider the corresponding set $\tilde{\S}_k = \{w_1(T_{a_1},T_{b_1}), \ldots, w_{n_k}(T_{a_1},T_{b_1})\}\subset B_3$ of representatives of the distinct cosets of $cap^{-1}(\Gamma_1(k))/\ker(cap)$. With this notation in place, we have the following.

\begin{theorem}
\label{thm:lmod_gen_set}
For $g,k \geq 2$, 
$$\Mod_{p_k}(S_g,e_1) = \langle \hat{i}(\tilde{\S}_k) \cup \S_{g,k} \cup \{T_{a_2},\ldots,T_{a_g},T_{b_2}, \ldots,T_{b_g},T_{c_1},\ldots,T_{c_{g-1}} \} \rangle.$$
\end{theorem} 

\begin{proof}
For simplicity, we will only consider the case when $g=2$, as our arguments easily generalize for any arbitrary $g \geq 3$. We consider the canonical embedding $\eta_k : \SL(2,\Z_k) \to \Sp(4,\mathbb{Z}_k)$ defined by 
$$A \xmapsto{\eta_k}  \begin{bmatrix}
A       & O_2\\
O_2 & I_2
\end{bmatrix},$$ where $O_2$ is the $2 \times 2$ zero block and $I_2$ is the $2 \times 2$ identity block. To prove our assertion, it suffices to show that
$$\Psi_k(\Mod_{p_k}(S_2, e_1))=\langle \eta_k(\Psi_k(\Mod_{p_k}(S_1,e_1))) \cup \{ \Psi_k(T_{c_1}),\Psi_k(T_{a_2}),\Psi_k(T_{b_2})\} \rangle,$$ where $$\Psi_k(T_{c_1})=\begin{bmatrix}
	1& 1& 0 & -1\\
	0& 1& 0 & 0\\
	0& -1& 1 & 1\\
	0& 0& 0 & 1
	\end{bmatrix}, \Psi_k(T_{a_2})=\begin{bmatrix}
	1& 0& 0 & 0\\
	0& 1& 0 & 0\\
	0& 0& 1 & 1\\
	0& 0& 0 & 1
	\end{bmatrix}, \text{ and } \Psi_k(T_{b_2})=\begin{bmatrix}
	1& 0& 0 & 0\\
	0& 1& 0 & 0\\
	0& 0& 1 & 0\\
	0& 0& -1 & 1
	\end{bmatrix}.$$
	
\noindent From Theorem~\ref{thm:symp_Gamma1_crit}, we know that an arbitrary matrix $A \in \Psi_k(\Mod_{p_k}(S_2,e_1))$ has the form
$$A=\begin{bmatrix}
	1& e_{12}& e_{13} & e_{14}\\
	0& 1 & 0 & 0\\
	0& e_{32} & e_{33} & e_{34}\\
	0& e_{42} & e_{43} & e_{44}
	\end{bmatrix},$$ where $e_{33}e_{44}-e_{34}e_{43} \equiv 1\pmod{k}$. We now consider the matrices
	$$M_2 = \begin{bmatrix}
	1 & -e_{12} & 0 & 0 \\
	0 & 1 & 0 & 0 \\
	0& 0 & 1 & 0\\
	0 & 0 & 0 & 1
	\end{bmatrix} \text{ and } M_1 = \begin{bmatrix}
	1 & 0 & 0 & 0 \\
	0 & 1  & 0 & 0 \\
	0 & 0 & e_{44} & -e_{34}\\
	0 & 0 & -e_{43} & e_{33}\\
	\end{bmatrix}$$
	in $\Sp(4,\Z_k)$. 
A direct computation reveals that 
	$$A\cdot M_1 \cdot M_2 = \begin{bmatrix}
	1& 0& \alpha & \beta\\
	0& 1 & 0 & 0\\
	0& \beta & 1 & 0\\
	0& -\alpha & 0 & 1
	\end{bmatrix},$$
	where $\alpha= e_{13}e_{44}-e_{14}e_{43}$ and $\beta=e_{14}e_{33}-e_{13}e_{34}$. Furthermore, by considering $$M_3=\Psi_k(T_{c_1}^\beta) \cdot \Psi_k(T_{a_2}^{-\beta})=\begin{bmatrix}
	1& \beta & 0 & -\beta\\
	0& 1 & 0 & 0\\
	0& -\beta & 1 & 0\\
	0& 0 & 0 & 1
	\end{bmatrix},$$ we see that 
	$$A\cdot M_1 \cdot M_2 \cdot M_3=\begin{bmatrix}
	1& \beta-\alpha\beta& \alpha & 0\\
	0& 1 & 0 & 0\\
	0& 0 & 1 & 0\\
	0& -\alpha & 0 & 1
	\end{bmatrix}.$$
Finally, taking
\begin{eqnarray*}
M_4 & = & (\Psi_k(T_{a_2})\cdot \Psi_k(T_{b_2})\cdot \Psi_k(T_{a_2}))^{-1} \cdot\Psi_k(T_{c_1}^{-\alpha})\cdot\Psi_k(T_{a_2}^\alpha) \cdot (\Psi_k(T_{a_2})\cdot \Psi_k(T_{b_2}) \cdot \Psi_k(T_{a_2})) \\ 
  &= &\begin{bmatrix}
	1& -\alpha& -\alpha & 0\\
	0& 1 & 0 & 0\\
	0& 0 & 1 & 0\\
	0& \alpha & 0 & 1
	\end{bmatrix},
	\end{eqnarray*}
	we get
	$$A\cdot M_1 \cdot M_2 \cdot M_3 \cdot M_4=\begin{bmatrix}
	1& \beta-\alpha-\alpha\beta& 0 & 0\\
	0& 1 & 0 & 0\\
	0& 0 & 1 & 0\\
	0& 0 & 0 & 1
	\end{bmatrix} \in \eta_k(\Psi_k(\Mod_{p_k}(S_1,e_1))),$$ thereby proving our assertion.
\end{proof}

\noindent Theorem~\ref{thm:lmod_gen_set} together with Corollary~\ref{cor:quot_lmod} yields the following.

\begin{corollary}
\label{cor:lmod_gen_set2}
For $g,k \geq 2$, 
$$\LMod_{p_k}(S_g) = \langle \S_k'' \cup \hat{i}(\tilde{\S}_k) \cup \S_{g,k} \cup \{T_{a_2},\ldots,T_{a_g},T_{b_2}, \ldots,T_{b_g},T_{c_1},\ldots,T_{c_{g-1}} \} \rangle.$$
\end{corollary}

\noindent It is not hard to see that 
$\cap_{k \geq 2} \, \Mod_{p_k}(S_g,e_1) = \Stab_{\Mod(S_g)}(e_1)$ is a index $2$ subgroup of $\UMod(S_g)$, where the $e_1$ appearing in $\Stab_{\Mod(S_g)}(e_1)$ refers to $e_1 \in H_1(S_g, \Z)$. By applying Corollary~\ref{cor:lmod_gen_set2} and using arguments similar to the ones used in the proof of Theorem~\ref{thm:lmod_gen_set}, we obtain the following. 

\begin{corollary}
\label{cor:Umod_gens}
For $g \geq 1$ and $k \geq 2$, 
$$\UMod(S_g) = \langle \S(g) \cup \{T_{a_1},T_{a_2},\ldots,T_{a_g},T_{b_2}, \ldots,T_{b_g},T_{c_1},\ldots,T_{c_{g-1}}, \iota \} \rangle,$$ where $\I(S_g) = \langle \S(g) \rangle$.
\end{corollary}

\noindent In particular, Corollary~\ref{cor:Umod_gens} implies that $\UMod(S_g)$ is finitely generated for $g \geq 3$, as $\I(S_g)$ is known~\cite{DJ} to be finitely generated for $g \geq 3$. An interesting question that arises in this context is whether $\UMod(S_2)$ is finitely generated. We plan to investigate this in future works.

\subsection{Finite order and reducible maps in $\Mod_{p_k}(S_g,e_1)$} 
By the Nielsen realization theorem~\cite{SK,JN}, we know that an $f \in \Mod(S_g)$ of finite order is represented by a $F \in \Homeo^+(S_g)$ of the same order. Hence, we can associate a \textit{corresponding orbifold} $\O_f := S_g/\langle F \rangle$ to $f$, whose genus we denote by $g(\O_f)$. 

\begin{proposition}
\label{prop:fin_ord_orb_gen_pos}
Consider an $f \in \Mod(S_g)$ of finite order such that $g(\O_f) > 0$. Then $f$ has a conjugate in $ \Mod_{p_k}(S_g,e_1)$. 
\end{proposition} 

\begin{proof}
Consider an $f \in \Mod(S_g)$ of order $n$ such that $g(\O_f) > 0$. Then there exists a nonseparating curve $c \in \O_f$, whose preimage under the branched covering $S_g \to \O_f$ is a reduction system $\C = \{c_1,\ldots,c_n\}$ for $f$ comprising nonseparating curves that are cyclically permuted by $f$. Thus, $f$ stabilizes the primitive homology vector $v = c_1+\ldots + c_k \in H_1(S_g,\Z_k)$. Since $\Mod(S_g)$ acts transitively on the primitive vectors of $H_1(S_g,\Z_k)$, there exists a $\psi \in \Mod(S_g)$ such that 
$\psi(v) = e_1$. Thus, it follows that $f$ has a conjugate $ \psi f \psi^{-1} \in \Mod_{p_k}(S_g,e_1)$, as desired.
\end{proof}

\noindent A variant of this result was proved in~\cite{DR}, where it was also shown that any map as described in Proposition~\ref{prop:fin_ord_orb_gen_pos} lifts to a map in $\Mod(S_{g_k})$ that commutes with the generator of the deck transformation group of our cover $p_k : S_{g_k} \to S_g$. 

It is well known~\cite{JG3} that a finite order mapping class is irreducible if and only if $\O_f$ is a sphere with three cone points.

\begin{proposition}
For $g,k \geq 2$, let $f \in \Mod(S_g)$ be an irreducible mapping class of finite order. Then $f \notin \Mod_{p_k}(S_g,e_1)$.
\end{proposition}

\begin{proof}
Let us assume on the contrary that $f \in \Mod_{p_k}(S_g,e_1)$. We consider an arbitrary representative $w \in H_1(S_g,\Z)$ of $e_1 \in H_1(S_g,\Z_k)$ and its $\langle f \rangle $-orbit $\mathbb{O} = \{w,f(w),\ldots,f^{k(w)-1}(w)\}$, where $k(w) \mid |f|$. Then $f$ preserves the vector $V_{\mathbb{O}}= \sum_{i=0}^{k(w)-1} f^i(w) \in H_1(S_g,\Z)$. As the $\langle f \rangle$-action on set of representative vectors in $e_1$ partitions it into infinitely many disjoint orbits, it must preserve infinitely many distinct vectors of the form $V_{\mathbb{O}}$. Thus, we conclude that $f \in \I(S_g)$. This contradicts our assumption, as $\I(S_g)$ is torsion-free for $g \geq 1$ (see~\cite[Theorem 6.12]{FM}).
\end{proof}

\begin{remark}\label{rem:dehn_roots}
Let $c$ be a simple closed curve in $S_g$, for $g \geq 2$. It is well known~\cite{MK1,KR1} that a root $f$ of $T_c$ of degree $n$ preserves $c$ and induces an order $n$ map $\hat{f}$ on the $\widehat{S}_g(c)$ obtained by capping off the boundary components of $\overline{S_g \setminus c}$. When $c$ is a nonseparating curve, $\widehat{S}_g(c) \approx S_{g-1}$, and $\hat{f} \in \Mod(S_{g-1})$. However, when $c$ is separating, $\widehat{S}_g(c) \approx S_{g_1} \sqcup S_{g_2}$, where $S_g = S_{g_1} \#_c S_{g_2}$, $\hat{f}_i := \hat{f}\vert{S_{g_i}} \in \Mod(S_{g_i})$ for $i = 1,2$, and $\text{lcm}(|\hat{f}_1|,|\hat{f}_2|)=n$. 
\end{remark}

\noindent We now have the following assertion. 
\begin{proposition}
\label{prop:root_lift}
Let $f \in \Mod(S_g)$ be a root of a $T_c$ of degree $n$. 
\begin{enumerate}[(i)]
\item If $c$ is nonseparating, then $f$ has a conjugate in $\Mod_{p_k}(S_g,e_1)$. 
\item If $c$ is separating and $g(\O_{\hat{f}_i})>0$ for at least one $i \in \{1,2\}$, then $f$ has a conjugate in $\Mod_{p_k}(S_g,e_1)$.
\end{enumerate}
\end{proposition} 
\begin{proof} 
By Remark~\ref{rem:dehn_roots}, we have that $f(c) = c$, and so $f \in \text{Stab}_{\Mod(S_g)}(c)$, where we view $c$ as a primitive vector in $H_1(S_g,\Z_k)$. Now consider a $\psi \in \Mod(S_g)$ such that $\psi(c) = a_1$. Then a fundamental property of Dehn twists would imply that
$$\psi T_c \psi^{-1} = T_{\psi(c)} = T_{a_1}.$$
Since $f^n = T_c$, we have 
$$(\psi f \psi^{-1})^n = \psi f^n \psi^{-1} = \psi T_c \psi^{-1} = T_{a_1}.$$ Thus, $\psi f \psi^{-1}$ is a root of $T_{a_1}$, and so we have $(\psi f \psi^{-1})(a_1) = a_1$, from which (i) follows.

The assertion in (ii) follows directly from Remark~\ref{rem:dehn_roots} and  Proposition~\ref{prop:fin_ord_orb_gen_pos}.
\end{proof}

\subsection{A class of liftable Penner-type pseudo-Anosovs}  Let $\lambda(f)$ denote the stretch factor of the pseudo-Anosov mapping class $f$. Generalizing a result of Thurston, Penner~\cite{RP} gave the following recipe for constructing pseudo-Anosovs in $\Mod(S_g)$.
\begin{theorem}
\label{thm:pA_recipe}
Let $\C = \{\alpha_1,\ldots,\alpha_n\}$ and $\D = \{\alpha_{n+1},\ldots,\alpha_{n+m}\}$ be multicurves in $S_g$ that together fill $S_g$.  
\begin{enumerate}[(i)]
\item Then any product of positive powers of the $T_{\alpha_i}$, for $i=1,\ldots,n$ and negative powers of the $T_{\alpha_{n+j}}$, for $j=1,\ldots,m$, where each $\alpha_i$ and each $\alpha_{n+j}$ appears at least once, is pseudo-Anosov.
\item To each $T_{\alpha_{k}}$, associate a matrix $B_k=I+A_k \Sigma$, where $I$ is the identity matrix of order $n+m$, $A_k$ is the $(n+m) \times (n+m)$ matrix all of whose entries are zero, except for the $(kk)^{th}$ entry, which is $1$, and $\Sigma$ is the incidence matrix of the imbedded graph $\C \cup \D$. For a pseudo-Anosov $f=T_{\alpha_{i_1}}\ldots T_{\alpha_{i_s}}$ as in (i), $\lambda(f)$ is the largest eigenvalue of the Perron matrix 
\[
M_f := B_{i_1}\ldots B_{i_s}.
\]
\end{enumerate}
\end{theorem}

\noindent  For $g \geq 1$, we consider maps in $\Mod(S_g)$ of the following form. 

\begin{definition}
\label{def:pa_family}
Given an integer $g \geq 2$, consider a tuple of positive integers of the form 
$$\T = ((p_1,\ldots,p_{g-1}), ((k_1,\ell_1),\ldots,(k_g,\ell_g)),$$ which we call
an \textit{admissible tuple $\T$ of genus g}.  Then given an admissible $\T$ of genus $g$, we define  $h_\T \in \Mod(S_g)$ by 
$$h_\T := \prod_{i=1}^{g-1} T_{c_i}^{-p_i} \prod_{j=1}^g \left(T_{a_j}^{-k_j}T_{b_j}^{\ell_j}\right).$$
\end{definition} 

\noindent In the following lemma, we derive some basic properties of the map $h_{\T}$. 

\begin{lemma}\label{mainthm}
Consider a map $h_\T$ as in Definition~\ref{def:pa_family}. Then:
\begin{enumerate}[(i)]
\item $h_\T$ is a pseudo-Anosov map, and 
\item $\Mod(S_g)$ is generated by $3g$ elements in $\{h_\T:\T \text{ is an admissible tuple of genus g}\}$.
\end{enumerate}
\end{lemma}
\begin{proof}
Consider the multicurves $\mathcal{C}=\{c_1,\cdots,c_{g-1},a_1,a_2,\cdots,a_g\}$ and $\mathcal{D}=\{b_1,b_2,\cdots,b_g\}$ in $S_g$. Since the collection $\C\cup \D$ fills $S_g$, (i) follows directly from Theorem~\ref{thm:pA_recipe}. 

To show (ii), we first relabel the curves in $\C \cup \D$ as
$$(c_1,\cdots,c_{g-1},a_1,\cdots,a_g, b_1,\ldots,b_g) = (\alpha_1,\ldots\alpha_{3g-1}),$$ and write
$\displaystyle h_{\T} = \prod_{i=1}^{3g-1} T_{\alpha_i}^{\delta_i(\T)}.$
Then $T_{\alpha_1} = h_{\T_{\alpha_1}}^{-1}h_{\T_1}$, where
$$\T_1 = ((1,\ldots,1),((1,1),\ldots,(1,1)) \text{ and } \T_{\alpha_1} = (2,1,\ldots,1),((1,1),\ldots,(1,1)).$$ 
Proceeding inductively, we see that for $k>1$
$$T_{\alpha_k}^{\delta_k} = \left( \prod_{i=1}^{k-1}T_{\alpha_i}^{\delta_{i}} \right)^{-1} h_{\T_{\alpha_k}}^{-1} \left( \prod_{i=1}^{k-1}T_{\alpha_i} \right) h_{\T_1},$$ where $\delta_i = \delta_i(\T_1)$ and $\T_{\alpha_k}$ is the admissible tuple obtained by replacing the $k^{th}$ appearance of $1$ in $\T_1$ with $2$.  Since the set of $3g-1$ Dehn twists $\{T_\alpha: \alpha \in \C \cup \D\}$ generate $\Mod(S_g)$, it follows that $\{h_{\T_{\alpha_k}} : 1 \leq k \leq 3g-1\} \cup \{\T_1\}$ generates $\Mod(S_g)$.
\end{proof}

\begin{remark}
Let $P(A)$ denote the characteristic polynomial of a square matrix $A$. A direct computation reveals that $P(\Psi(h_{\T})) = (x-1)P(M_{h_{\T}})$, where $M_{h_{\T}}$ is the Perron matrix of $h_{\T}$ (as in Theorem~\ref{thm:pA_recipe}). Hence, the homological dilatation of $h_{\T}$ equals its stretch factor. Consequently, by a result of Silver~\cite{SW}, it follows that  $h_{\T}$ induces  orientable foliations on $S_g$.
\end{remark}

\noindent Using our symplectic criterion for liftability (Theorem~\ref{thm:symp_lift_crit}), we will now derive a necessary and sufficient condition under which a pseudo-Anosov map of type $h_{\T}$ would lie in $\LMod_{p_k}(S_g)$.

\begin{proposition}
For $g \geq 2$,  let $h_{\T} \in \Mod(S_g)$ be a pseudo-Anosov mapping class as in Definition~\ref{def:pa_family}. Then $h_{\T} \in \LMod_{p_k}(S_g)$ if and only if $k \mid \ell_1$. 
\end{proposition}

\begin{proof}
Suppose that $k \mid \ell_1$. Then by Corollary~\ref{cor:lmod_gen_set2}, $h_{\T} \in \LMod_{p_k}(S_g)$. Let $O_m$ and $I_m$, respectively, denote the $m \times m$ zero matrix and the $m\times m$ identity matrix. For showing the converse, we first note that for $1 \leq i \leq g$, the matrices $\Psi(T_{a_i}^{-1})$ and $\Psi(T_{b_i})$ in $\Sp(2g,\Z)$, respectively, have the following $g \times g$ block matrix structure (comprising $2 \times 2$ blocks):
\[
\small
\Psi(T_{a_i}) = \begin{bmatrix}
I_{2}&O_{2}&\cdots&\cdots&\cdots & O_{2}\\
\ddots\\
O_{2}&O_{2}&\cdots& A &\cdots & O_{2}\\
\ddots\\
O_{2}&O_{2}&\cdots&\cdots&\cdots & I_{2}
\end{bmatrix} \text{ and } 
\Psi(T_{b_i}^{-1}) =\begin{bmatrix}
I_{2}&O_{2}&\cdots&\cdots&\cdots & O_{2}\\
\ddots\\
O_{2}&O_{2}&\cdots& B &\cdots & O_{2}\\
\ddots\\
O_{2}&O_{2}&\cdots&\cdots&\cdots & I_{2}
\end{bmatrix},
\]
where the $A=\begin{bmatrix}1&1\\0&1 \end{bmatrix}$, and $B=\begin{bmatrix}1&0\\1&1 \end{bmatrix}$ are $(ii)^{th}$ blocks of the corresponding matrices. Thus, we have 
\[\small
\Psi(T_{a_i}^{-k_j}T_{b_i}^{\ell_j})= \begin{bmatrix}
I_{2}&O_{2}&\cdots&\cdots&\cdots & O_{2}\\
\ddots\\
O_{2}&O_{2}&\cdots& X_{k_j,\ell_j} &\cdots & O_{2}\\
\ddots\\
O_{2}&O_{2}&\cdots&\cdots&\cdots & I_{2}
\end{bmatrix},
\]
where  the $(ii)^{th}$ block $X_{k_j,\ell_j} = \begin{bmatrix}1+k_j\ell_j&k_j\\\ell_j&1 \end{bmatrix}$. Moreover, $\Psi(T_{c_i}^{-1})$, for $1\leq i\leq g-1$, is the block matrix
\[ \small
\Psi(T_{c_i}^{-1}) = \begin{bmatrix}
I_{2}&O_{2}&\cdots&\cdots&\cdots & \cdots& O_{2}\\
\ddots\\
O_{2}&O_{2}& \cdots&P&R &\cdots& O_{2}\\
O_{2}&O_{2}& \cdots&R&P &\cdots& O_{2}\\
\ddots\\
O_{2}&O_{2}&\cdots&\cdots&\cdots & \cdots& I_{2}\\
\end{bmatrix},
\]
whose $(ii)^{th}$ and $((i+1)(i+1))^{th}$ blocks are $P=\begin{bmatrix}1&1\\0&1 \end{bmatrix}$, and whose $(i(i+1))^{th}$ and $((i+1)i)^{th}$  $i\times i$ blocks are $R=\begin{bmatrix}0&-1\\0&0 \end{bmatrix}$. Consequently, we have $\Psi(T_{c_i}^{p_i})$ is the block matrix
\[\small
\Psi(T_{c_i}^{p_i}) =\begin{bmatrix}
I_{2}&O_{2}&\cdots&\cdots&\cdots & \cdots& O_{2}\\
\ddots\\
O_{2}&O_{2}& \cdots&P_{p_i}&R_{p_i} &\cdots& O_{2}\\
O_{2}&O_{2}& \cdots&R_{p_i}&P_{p_i} &\cdots& O_{2}\\
\ddots\\
O_{2}&O_{2}&\cdots&\cdots&\cdots & \cdots& I_{2}\\
\end{bmatrix},
\]
whose $(ii)^{th}$ and $((i+1)(i+1))^{th}$ blocks are $P_{p_i}=\begin{bmatrix}1&p_i\\0&1 \end{bmatrix}$, and whose $(i(i+1))^{th}$ and $((i+1)i)^{th}$  $i\times i$ blocks are $R_{p_i}=\begin{bmatrix}0&-p_i\\0&0 \end{bmatrix}$. Further, we note that $P_{p_i}^2=P_{2p_i}$, $R_{p_i}^2=O_2$, $P_{p_i}R_{p_i}=R_{p_i}$, and $R_{p_i}P_{p_i}=R_{p_i}$. Thus, we see  that $\Psi(h_{\T}) = (d_{ij})_{2g \times 2g}$, where $d_{21} = \ell_1$. Hence, by Theorem~\ref{thm:symp_lift_crit}, it follows $k \mid \ell_1$.
\end{proof}

\section*{Acknowledgements} The authors would like to thank Prof. Dan Margalit for inspiring some of the key ideas in this paper, Dr. Karam Deo Shankhadhar for providing pertinent references from the theory of modular forms, and Pankaj Kapari for pointing out some typos in an earlier version of the paper.

\bibliographystyle{plain}
\bibliography{LModSg}

\begin{thebibliography}{10}

\bibitem{HA}
Hiroki Aoki.
\newblock On {S}iegel paramodular forms of degree 2 with small levels.
\newblock {\em Internat. J. Math.}, 27(2):1650011, 20, 2016.

\bibitem{BH1}
Joan~S. Birman and Hugh~M. Hilden.
\newblock On the mapping class groups of closed surfaces as covering spaces.
\newblock In {\em Advances in the theory of {R}iemann surfaces ({P}roc.
  {C}onf., {S}tony {B}rook, {N}.{Y}., 1969)}, pages 81--115. Ann. of Math.
  Studies, No. 66, 1971.

\bibitem{BH2}
Joan~S. Birman and Hugh~M. Hilden.
\newblock Isotopies of homeomorphisms of {R}iemann surfaces and a theorem about
  {A}rtin's braid group.
\newblock {\em Bull. Amer. Math. Soc.}, 78:1002--1004, 1972.

\bibitem{BH3}
Joan~S. Birman and Hugh~M. Hilden.
\newblock On isotopies of homeomorphisms of {R}iemann surfaces.
\newblock {\em Ann. of Math. (2)}, 97:424--439, 1973.

\bibitem{DR}
Neeraj~K Dhanwani and Kashyap Rajeevsarathy.
\newblock Commuting conjugates of finite-order mapping classes.
\newblock {\em arXiv preprint arXiv:1901.11314}, 2019.

\bibitem{DKMO}
N.~Diamantis, M.~Knopp, G.~Mason, and C.~O'Sullivan.
\newblock {$L$}-functions of second-order cusp forms.
\newblock {\em Ramanujan J.}, 12(3):327--347, 2006.

\bibitem{FM}
Benson Farb and Dan Margalit.
\newblock {\em A primer on mapping class groups}, volume~49 of {\em Princeton
  Mathematical Series}.
\newblock Princeton University Press, Princeton, NJ, 2012.

\bibitem{NJF}
Neil~J. Fullarton.
\newblock A generating set for the palindromic {T}orelli group.
\newblock {\em Algebr. Geom. Topol.}, 15(6):3535--3567, 2015.

\bibitem{GW2}
Tyrone Ghaswala and Rebecca~R. Winarski.
\newblock The liftable mapping class group of balanced superelliptic covers.
\newblock {\em New York J. Math.}, 23:133--164, 2017.

\bibitem{GW1}
Tyrone Ghaswala and Rebecca~R. Winarski.
\newblock Lifting homeomorphisms and cyclic branched covers of spheres.
\newblock {\em Michigan Math. J.}, 66(4):885--890, 2017.

\bibitem{TG}
{Ghaswala, Tyrone}.
\newblock The liftable mapping class group, 2017.

\bibitem{JG3}
Jane Gilman.
\newblock Structures of elliptic irreducible subgroups of the modular group.
\newblock {\em Proc. London Math. Soc. (3)}, 47(1):27--42, 1983.

\bibitem{H1}
W.~J. Harvey.
\newblock Cyclic groups of automorphisms of a compact {R}iemann surface.
\newblock {\em Quart. J. Math. Oxford Ser. (2)}, 17:86--97, 1966.

\bibitem{IO}
T.~Ibukiyama and F.~Onodera.
\newblock On the graded ring of modular forms of the {S}iegel paramodular group
  of level {$2$}.
\newblock {\em Abh. Math. Sem. Univ. Hamburg}, 67:297--305, 1997.

\bibitem{DJ}
Dennis Johnson.
\newblock The structure of the {T}orelli group. {I}. {A} finite set of
  generators for $\mathcal{I}$.
\newblock {\em Ann. of Math. (2)}, 118(3):423--442, 1983.

\bibitem{SK}
Steven~P. Kerckhoff.
\newblock The {N}ielsen realization problem.
\newblock {\em Ann. of Math. (2)}, 117(2):235--265, 1983.

\bibitem{WBRL}
W.~B.~R. Lickorish.
\newblock A finite set of generators for the homeotopy group of a
  {$2$}-manifold.
\newblock {\em Proc. Cambridge Philos. Soc.}, 60:769--778, 1964.

\bibitem{M1}
Alexander~Murray Macbeath and HC~Wilkie.
\newblock {\em Discontinuous groups and birational transformations:[Summer
  School], Queen's College Dundee, University of St. Andrews}.
\newblock [Department of Math.], Queen's College, 1961.

\bibitem{MK1}
Darryl McCullough and Kashyap Rajeevsarathy.
\newblock Roots of {D}ehn twists.
\newblock {\em Geometriae Dedicata}, 151:397--409, 2011.
\newblock 10.1007/s10711-010-9541-4.

\bibitem{TM}
Toshitsune Miyake.
\newblock {\em Modular forms}.
\newblock Springer Monographs in Mathematics. Springer-Verlag, Berlin, english
  edition, 2006.
\newblock Translated from the 1976 Japanese original by Yoshitaka Maeda.

\bibitem{MP}
Michele Mulazzani and Riccardo Piergallini.
\newblock Lifting braids.
\newblock {\em arXiv preprint math/0107117}, 2001.

\bibitem{JN}
Jakob Nielsen.
\newblock Abbildungsklassen endlicher {O}rdnung.
\newblock {\em Acta Math.}, 75:23--115, 1943.

\bibitem{RP}
Robert~C. Penner.
\newblock A construction of pseudo-{A}nosov homeomorphisms.
\newblock {\em Trans. Amer. Math. Soc.}, 310(1):179--197, 1988.

\bibitem{KR1}
Kashyap Rajeevsarathy.
\newblock Roots of dehn twists about separating curves.
\newblock {\em Journal of the Australian Mathematical Society}, 95:266--288, 10
  2013.

\bibitem{RS}
Brooks Roberts and Ralf Schmidt.
\newblock On modular forms for the paramodular groups.
\newblock In {\em Automorphic forms and zeta functions}, pages 334--364. World
  Sci. Publ., Hackensack, NJ, 2006.

\bibitem{SW}
Daniel~S. Silver and Susan~G. Williams.
\newblock Lehmer's question, knots and surface dynamics.
\newblock {\em Math. Proc. Cambridge Philos. Soc.}, 143(3):649--661, 2007.

\bibitem{WS}
William Stein.
\newblock {\em Modular forms, a computational approach}, volume~79 of {\em
  Graduate Studies in Mathematics}.
\newblock American Mathematical Society, Providence, RI, 2007.
\newblock With an appendix by Paul E. Gunnells.

\end{thebibliography}

\end{document}